\documentclass[10pt,letterpaper,oneside]{amsart}
                      \usepackage[english]{babel}
\usepackage{titletoc}
\usepackage{geometry,amssymb,mathrsfs,hyperref,xcolor}
\usepackage{mathtools}
\usepackage{textcomp}
\usepackage{amsmath}
\usepackage{algorithm}
\usepackage{algpseudocode}
\usepackage{thmtools}
\usepackage{tikz}
\usepackage{tikz-cd}
\usepackage{graphicx}
\usepackage{comment}
\usepackage{enumitem}
\usepackage{colonequals}
\usepackage[backend=biber]{biblatex}
\newtheorem{theorem}{Theorem}[section]
{\theoremstyle{definition}\newtheorem{definition}[theorem]{Definition}}
\def\MR#1{}
{\theoremstyle{remark}}

\newtheorem{lemma}[theorem]{Lemma}

\newtheorem{conjecture}[theorem]{Conjecture}
{\theoremstyle{definition}
\newtheorem{myalgorithm}[theorem]{Algorithm}}
\newcommand{\acknowledgement}{\subsection*{Acknowledgements}}

\author{Antonio Rieser}
\date{}
\title{A New Construction of the Vietoris-Rips Complex}
\hypersetup{
 pdfauthor={Antonio Rieser},
 pdftitle={A New Construction of the Vietoris-Rips Complex},
 pdfkeywords={},
 pdfsubject={},
 pdfcreator={Emacs 27.1 (Org mode 9.6)}, 
 pdflang={English}}

\bibliography{/home/rieser/Bib/all}

\thanks{This work was partially supported by the
	CONACYT Investigadoras y Investigadores por M{\'e}xico Project \#1076,
	the grant N62909-19-1-2134 from the US Office of Naval Research Global and the	Southern Office of Aerospace Research and Development of the US Air Force
	Office of Scientific Research, the National Science Foundation Grant No. DMS-1929284 while the author was in residence at the Institute for Computational and Experimental Research in Mathematics in Providence, RI, during the "Math + Neuroscience: Strengthening the Interplay Between Theory and Mathematics" program, and the National Science Foundation Grant No. DMS-1928930 while the author was in residence at the Mathematical Sciences Research Institute (MSRI), now the Simons Laufer Mathematical Sciences Institute (SLMath).}

\begin{document}

\begin{abstract}
	We present a new, inductive construction of the Vietoris-Rips complex, 
	in which we take advantage of a small amount of unexploited 
	combinatorial structure in the $k$-skeleton of the complex in order to 
	avoid unnecessary comparisons when identifying its $(k+1)$-simplices. 
	In doing so, we achieve a significant reduction in the number of 
	comparisons required to construct the Vietoris-Rips compared to 
	state-of-the-art algorithms, which is seen here by examining the computational complexity of the critical step in the algorithms. In experiments comparing a C/C++ implementation of our algorithm to the GUDHI v3.9.0 software package \cite{gudhi:RipsComplex}, this results in an observed $5$-$10$-fold 
	improvement in speed of on sufficiently sparse Erd\H{o}s-Rényi graphs 
	with the best advantages as the graphs become sparser, as well as for higher 
	dimensional Vietoris-Rips complexes. We further clarify that the algorithm described in \cite{Boissonnat_Maria_2012} for the 
	construction of the Vietoris-Rips complex is exactly the Incremental 
	Algorithm from \cite{Zomorodian_2010}, albeit with the additional 
	requirement that the result be stored in a tree structure, and we 
	explain how these techniques are different from the algorithm 
	presented here.
\end{abstract}

\maketitle

\section{Introduction}
\label{sec:orgfff05a2}

The Vietoris-Rips complex of a point cloud, or, equivalently, the clique
complex of a graph, is the canonical tool for the calculating homology groups
associated to a point cloud, and has been since nearly the beginning of the
emergence of topological data analysis. Needless to say, using the best possible
algorithm for the construction of the complex is of paramount importance,
particularly since the construction of these complexes has a worst-case
complexity which is exponential in the number of points, and even the complexity of constructing the $k$-skeleton has a worst-case lower bound of $O(n^{(k+1)})$, where $n$ is the number of vertices in the complex. (See, for instance,
\cite{Lloyd_etal_2016} for a discussion of the complexity of computing homology from
a point cloud, in addition to a presentation of a quantum algorithm which provides an exponential improvement over the worst-case classical scenario.) Fortunately, however, the computation of clique
complexes for sparse graphs may be considerably faster in practice, particularly when one only
needs to compute a low-dimensional skeleton of the complex, as is the usual
practice in topological data analysis. Nonetheless, even with this restriction,
computing the persistent homology of large data sets tends to be rather
slow, limiting its practical application. While computing Euler characteristic curves \cites{Richardson_Werman_2014, Smith_Zavala_2021,Dlotko_Gurnari_arXiv_2022} and the Euler characteristic transform \cite{Amezquita_etal_2021} is
considerably faster than computing persistent homology, this, too, requires the initial construction of a filtered simplicial
complex when the original input data is in the form of a point cloud. It is
therefore essential to develop techniques and software designed to speed up the
construction of the simplicial complexes used in topological data analysis, with the ultimate goal of making
topological techniques more tractable for a wider range of data sets.

To the best of the author's knowledge, the algorithms described in \cite{Boissonnat_Maria_2012} and \cite{Zomorodian_2010}, and, in particular, the implementation of the Incremental-VR algorithm from \cite{Zomorodian_2010} in the GUDHI software package \cite{gudhi:RipsComplex}, represent the current state-of-the-art for the
construction of the Vietoris-Rips complex. In \cite{Zomorodian_2010}, Zomorodian
proposes three algorithms: first, an intuitive, inductive construction of the Vietoris-Rips complex (the Inductive-VR algorithm); second, the Incremental-VR algorithm, which is a refinement of the Inductive-VR algorithm designed to
avoid some repetitive calculations performed using its more naive predecessor; and, finally, a construction of the Vietoris-Rips complex which proceeds by finding the maximal cliques of a
graph, the Maximal-VR algorithm. He then studies their relative performance on a number of different data sets, and finds that, on most of the data sets, the Incremental-VR Algorithm indeed outperforms the others, although with the important caveat that, unlike the others, the Maximal-VR algorithm's performance may be improved through parallelization, even if its single-threaded
performance is significantly slower than the other two. In \cite{Boissonnat_Maria_2012}, Boissonnat and Maria introduce the simplex tree data structure, and present an algorithm that is equivalent to the Incremental-VR algorithm of \cite{Zomorodian_2010}, with the additional requirement that the simplicial complex be stored in a tree. While unreported in \cite{Boissonnat_Maria_2012}, the implementation of the Incremental-VR algorithm in GUDHI v3.9.0 uses a merge algorithm to compute the intersection in the critical step, which we briefly describe in Section \ref{sec:Background} below. One or another
of these algorithms have been implemented in almost every topological analysis
software package which has been written to date, a non-exhaustive list of which is given in \cites{Otter_etal_2017}.

In this article, we present a new algorithm for the construction of the
Vietoris-Rips complex in which the combinatorial structure of the complex is
exploited to avoid a large number of unnecessary calculations. We discuss the computational complexity of the critical step in constructing the $k$-skeleton from the $(k-1)$-skeleton when applied to Erd\H{o}s-Rényi graphs $G(n,p)$, showing that this step in the New-VR algorithm has complexity $O(np^{(k-1)})$, whereas the corresponding step in the best Incremental-VR algorithm as implemented in GUDHI v3.9.0 is expected to have an average complexity of least $O(pn)$, resulting in an overall advantage for the New-VR algorithm, particularly when $p$ is small and the dimension is large. We then present
several experiments where we compute the Vietoris-Rips complexes of
Erd\H{o}s-Rényi graphs using both the Incremental-VR algorithm and our New-VR algorithm, in which our algorithm performs approximately as expected: nearly ten times faster for sufficiently sparse Erd\H{o}s-Rényi graphs when compared to the Vietoris-Rips complex construction in GUDHI v3.9.0, with at least a $5$-fold improvement in most of the regimes which we tested. (The values of $p$ which we were able to take were unfortunately limited by memory constraints.) While the current implementations of our algorithm are slower than GUDHI v3.9.0 for sufficiently dense graphs in dimensions $2$ and $3$ and for complete graphs, we believe that the superior performance achieved by GUDHI v3.9.0 in these regimes to be the result of optimizations in other parts of the GUDHI v3.9.0 codebase (which is highly optimized relative to the experimental code to which it was compared) and not in the Vietoris-Rips construction itself. Indeed, this hypothesis supported by the observation that the GUDHI v3.9.0 code outperforms ours on complete graphs, when the number of comparisons performed by both algorithms should be the same. In dimensions higher than $3$, our software consistently outperforms GUDHI v3.9.0 in speed, with the exception of trials on complete graphs and those which were executed shortly before the program was killed for lack of memory. We discuss these issues further in Section \ref{sec:org4363633}. 

The algorithm we introduce depends on two
elementary observations: First, given two distinct \(n\)-cliques \(\sigma_0\)
and \(\sigma_1\) in a graph \(G\) which share a common \((n-1)\)-clique
\(\rho\), to determine whether \(\tau \coloneqq \sigma_0 \cup \sigma_1\) is itself an
\((n+1)\)-clique, we only need to verify the existence of a single additional
edge among the pairs of vertices in \(\tau\), since we know
that the rest of the edges are contained in one or the other of \(\sigma_0\) and $\sigma_1$. This edge,
furthermore, is easily identified to be the pair \(\{v_0,v_1\}\), where \(v_i =
\tau \backslash \sigma_i\) for \(i=0,1\). Second, every
\((n+1)\)-clique in a graph can be uniquely identified by the pair of
\(n\)-subcliques which is minimal among all such pairs with respect to the
lexicographic ordering on the pair and their vertices. (A precise definition of
this ordering, as well as the exact statements and proofs of the above assertions,
are given in the Sections \ref{sec:org4a1da2e} and \ref{sec:orgfb267b9}.) By taking advantage of the additional
structure provided by these observations, we are able reduce the
number of computations necessary to construct the clique complex, resulting in
an increase in the speed of the algorithm. Additionally, we mention how the Incremental-VR algorithm may be (embarrassingly) parallelized, and explain how our algorithm is similarly embarrassingly parallelizable. We naturally expect that parallelizing the implementation will lead to further improvements in speed.

\section{Combinatorial preliminaries}

\label{sec:org4a1da2e}

We now give the statements and proofs of the combinatorial facts
which underpin our construction of the Vietoris-Rips complex. The results are
entirely elementary, but we find their simplicity appealing, and we give their
proofs in full. We begin with the following lemma.

\begin{lemma}
\label{lem:Missing pair}
Let \(S\) be a set with \(|S| = n+1\) for some natural number \(n \geq 2\), and suppose that
the subsets \(S_0,S_1 \subset S\) satisfy \(|S_i| = n\) and that \(S_0 \neq
S_1\). Let \(t_i = S\backslash S_i\) for \(i=0,1\). Then the only pair \(\{a,b\} \subset S\)
which is not contained in one of the \(S_i\) is \(\{t_0,t_1\}\).
\end{lemma}

\begin{proof}
Without loss of generality, let \(S = \{ t_i \mid i\in \{0,\dots,n\}\}\). Then
\(S_0 = \{\hat{t}_0,t_1,t_2, \dots, t_n\}\), and \(S_0 =
\{t_0,\hat{t}_1,t_2,\dots,t_n\}\), where the notation \(\hat{t}_i\) indicates that the
element \(t_i\) is not in the set. Now note that every pair \(\{t_i,t_j\}, i,j\in
\{2,\dots\}\) is a subset of both \(S_0\) and \(S_1\). We also have that every
pair of the form \(\{t_0,t_i\}, i \in \{2,\dots,n\}\) is a subset of \(S_1\),
and every pair \(\{t_1,t_i\}, i \in \{2,\dots,n\}\) is a subset of \(S_0\). We
now see that the only pair contained in \(S\) but in neither of the \(S_i\) is
\(\{t_0,t_1\}\).
\end{proof}

\begin{theorem}
	\label{thm:Pair check}
Let \(G = (V,E)\) be an undirected graph, and let \(\Sigma_G\) denote the clique
complex of \(G\). Suppose, furthermore, that there exist \(k\)-simplices
\(\sigma_0,\sigma_1\), \(\sigma_0\neq\sigma_1\), which share the \((k-1)\)-face
\(\rho\). Let \(v_i = \sigma_i \backslash \rho\) for \(i=0,1\). If
\(\{v_0,v_1\} \in E\), then \(\sigma_0 \cup \sigma_1\) is a \((k+1)\)-simplex of
\(\Sigma_G\).
\end{theorem}

\begin{proof}
Without loss of generality, let \(\tau \coloneqq \{v_2,\dots,v_k\}\), and by
hypothesis, we have that \(\rho = \sigma_0 \cap \sigma_1\). The theorem now
follows directly from \autoref{lem:Missing pair} with \(S = \sigma_0 \cup
\sigma_1\), and \(S_i = \sigma_i\).
\end{proof}

The above theorem gives us a criterion for determining when two \(k\)-simplices
form two faces of a \((k+1)\)-simplex, but since each \((k+1)\)-simplex has many
\(k\)-faces, if we apply the criterion naively, we risk constructing the same
\((k+1)\)-simplex many times over unnecessarily. We will therefore need a bookkeeping mechanism on the \(k\)-simplices which will allow us to track which \((k+1)\)-simplices we have already created, and which candidates we still need to check. To this end, we introduce
an order on pairs of simplices of the same dimension which will enable us
to identify a \((k+1)\)-simplex with its minimal pair of \(k\)-faces. We begin with the following definition.

\begin{definition}
Let \(\Sigma\) be a simplicial complex, and endow the vertices of \(\Sigma\)
with a total order. For every dimension \(k \geq 0\) of \(\Sigma\),
we give the set of \(k\)-simplices the lexicographic order on their ordered vertices, i.e. \(\sigma_0 < \sigma_1\) iff the ordered
list of vertices of \(\sigma_0\) is less than the ordered list of vertices of \(\sigma_1\)
in the lexicographic order. We then define the relation
\(<\) on pairs of distinct \(k\)-simplices \(\{\sigma_0,\sigma_1\}\), \(\{\tau_0,\tau_1 \}\),\(\sigma_0 \neq \sigma_1\),\(\tau_0 \neq \tau_1\)
in a similar way. That is, first, we order each pair so that, without loss of generality,
we may assume that \(\sigma_0 < \sigma_1\) and that \(\tau_0 < \tau_1\). We will
write such a pair \((\sigma_0 < \sigma_1)\). We then apply 
the lexicographic order to the set of such pairs, i.e.
\begin{equation*}
(\sigma_0 < \sigma_1) < (\tau_0 < \tau_1) \text{ iff } \begin{cases}
\sigma_0 < \tau_0, \text{ or }\\
\sigma_0 = \tau_0 \text{ and }\sigma_1 < \tau_1.
\end{cases}
\end{equation*} 
\end{definition}

We now have the following lemma.

\begin{lemma}
	\label{lem:Total order on pairs}
The relation \(<\) is a total order on the set of pairs of \(k\)-simplices for
each dimension \(k\geq 0\).
\end{lemma}

\begin{proof} We first recall that the vertices of \(\Sigma\) are totally ordered, which implies that the lexicographic order on the \(k\)-simplices is a total order. The result follows.
\end{proof}

The next theorem will be instrumental for our construction.

\begin{theorem}
	\label{thm:Main}
Let \(\Sigma\) be a simplicial complex, and for each \(k \geq 0\), let \(\Sigma_{k+1}\) denote the set of \((k+1)\)-simplices of \(\Sigma\). Furthermore, let
\(P_{k}\) be the set of distinct pairs \((\sigma_0 < \sigma_1)\), \(\sigma_0\neq \sigma_1\) of \(k\)-simplices of
\(\Sigma\) such that
\begin{enumerate}[label=(\alph*)]
\item \label{item:a} If \((\sigma_0<\sigma_1)\in P_{k}\), then \(\tau\coloneqq \sigma_0 \cup \sigma_1 \in \Sigma_{k+1}\).

\item \label{item:b} If a pair \((\sigma_0 < \sigma_1) \in P_{k}\), then \((\sigma_0 < \sigma_1)\) is
minimal among all pairs of \(k\)-faces of any common \((k+1)\)-dimensional coface \(\tau\). Equivalently,
\(\sigma_0\) and \(\sigma_1\) are the smallest and second- smallest \(k\)-faces of \(\tau\) in the lexicographic order on \(k\)-simplices of \(\Sigma\).
\end{enumerate}
Then
\begin{enumerate}
\item \label{item:1} For every \(k>1\), there is a bijection between \(P_{k}\) and \(\Sigma_{k+1}\).
\item \label{item:2} If \(k>1\), then for every pair \(( \sigma_0 < \sigma_1 ) \in P_{k}\), \(\sigma_0\) and \(\sigma_1\)
have a common \((k-1)\)-face \(\rho \in \Sigma\), and, furthermore \(\rho\) is the minimal \((k-1)\)-face of \(\tau\), and therefore also of
of both \(\sigma_0\) and \(\sigma_1\). 
\end{enumerate}
\end{theorem}

\begin{proof}
First, by Item \ref{item:a}, there is a well-defined map \(\Psi:P_{k+1}\to \Sigma_{k+1}\). By Item \ref{item:b} and \autoref{lem:Total order on pairs}, \(\Psi\) is injective, and by \autoref{lem:Total order on pairs}, for every \(\tau \in \Sigma_{k+1}\), there exists a minimal pair \((\sigma_0 < \sigma_1)\) such that \(\tau = \sigma_0 \cup \sigma_1\), and therefore \(\Psi\) is surjective as well. We have therefore proven Item \ref{item:1}.

To see Item \ref{item:2}, note that for any pair \((\sigma_0 < \sigma_1)\in P_k\), \(\sigma_0\) and \(\sigma_1\) are \(k\)-faces of a common \((k+1)\)-dimensional simplex by hypothesis, and therefore they share a common \((k-1)\)-face \(\rho\). Since the pair \(\sigma_0 < \sigma_1\) is minimal among pairs of \(\tau = \sigma_0 \cup \sigma_1\), then writing \(\tau = (t_0,\dots,t_{k+1})\), where \(i<j \implies t_i < t_j\), it follows that \(\sigma_0\) and \(\sigma_1\) must be of the form \(\sigma_0 = \{t_0,\dots,t_{k-1}\}\) and \(\sigma_1 = \{t_0,\dots,t_{k-2},\hat{t}_{k-1},t_k\}\). Therefore the common face of \(\sigma_0\) and \(\sigma_1\) is \(\sigma_0\cap \sigma_1 = \{t_0,\dots,t_{k-2}\}\), which is the minimal \((k-1)\)-face of \(\tau\), as well as the minimal \((k-1)\)-face of both \(\sigma_0\) and \(\sigma_1\). 
\end{proof}
\section{The \texttt{New-VR} Algorithm}
\label{sec:orgfb267b9}

We now describe our algorithm, which builds the simplex tree of the simplicial complex \(\Sigma_G\), where the latter is
seen as a partial order on its vertex set. In this and the following, we abuse notation and refer to 
nodes in the simplex tree both by the simplices they represent and by the vertex label they are given.

\begin{myalgorithm}
\label{alg:Serial} (Serial) Initialization: Let
\(G=(V,E)\) be a graph. Put a total order on the vertices 
of the graph \(G\). For each vertex \(v \in G\), construct 
a node at level \(0\) of
the simplex tree, and identify the edges \(E_v\) for which \(v\) is the smaller
of its vertices. Create a node at level \(1\) in the simplex tree for each edge
\(e\) in \(E_v\), and connect the node in the simplex tree corresponding to
\(v\) to every node corresponding to an edge \(e \in E_v\).

Inductive step: Suppose that we have constructed a subset \(D\) of the simplex tree of \(\Sigma\) which has layers numbered \(0, \dots, k\) and such that
and element \(\rho\) of the \((k-1)\)-st layer of \(D\) is connected to and
element \(\sigma\) of the \((k)\)-th layer of \(D\) iff \(\rho\) is the minimal
\((k-1)\)-face of \(\sigma\). We construct the \((k+1)\)-st layer of \(D\) in
the following way: For each \(\rho\in D(k-1)\), let \(L_\rho\subset D(k)\) be
the list of elements in \(D(k)\) connected to \(\rho\) in \(D\), i.e. \(L_\rho\)
is the list of \(k\)-simplices \(\sigma \in \Sigma_G\) such that \(\rho\) is the
minimal face of \(\sigma\). For every ordered pair \((\sigma_0 < \sigma_1)\) of
simplices in \(L_\rho\), check whether the pair \(\{v_0,v_1\}\) forms an edge in
\(G\), where \(v_i \coloneqq \sigma_0\cup \sigma_1 \backslash \sigma_i\). If
\(\{v_0,v_1\}\) is an edge in \(G\), then we add the simplex \(\tau \coloneqq\)
to the \((k+1)\)-st layer of \(D\), and we connect \(\sigma_0\) to \(\tau\).

Note that, by construction, \(\sigma_0\) is the smallest \(k\)-face of the newly created simplex \(\tau\)
with respect to the lexicographic order.
\end{myalgorithm}

The serial algorithm may be summarized by the following four methods.
We begin with a preliminary function which finds the neighbors of a vertex 
$v$ in the graph $G$ whose label is larger than that of $v$.

\begin{algorithm}
	\caption{\texttt{Upper-Neighbors($G=(V,E),u\in V$)}}
	\begin{algorithmic}
		\State \Return {$\{v \in V \mid u < v, \{u,v\}\in E\}$}
	\end{algorithmic}	
\end{algorithm}

The next function, \texttt{Table-Lookup} is the function which creates the list of children $M$ of a vertex $v$ in the simplex tree from an ordered list $L$ of the siblings of $g$ and the (adjacency matrix of the) graph $G$.  Note that, in line \ref{line:stopping cond TL}, we may simply consider all elements in $N$, but including the condition that $w$ be also less than the largest neighbor of $v$ slightly reduces the speed of the implementation in practice, and similar conditions are included in the corresponding section of GUDHI v3.9.0. 

\begin{algorithm}
	\caption{\texttt{Table-Lookup}($G,N,L,v$)}
	\begin{algorithmic}[1]
		\State $M \gets \texttt{empty list}$
		\ForAll{\texttt{($w \leq \min\{\text{end}(N),L(v)\}$) }} \label{line:stopping cond TL}
		\If{$\{v,w\}\in E$}
		\State $M \gets M \cup \{w\}$
		\EndIf 
		\EndFor
		\State \Return M
	\end{algorithmic}
\end{algorithm}

 The function \texttt{New-Add-Cofaces} takes a simplicial complex $\Sigma$, assumed to be the clique complex of the subgraph of $G={V,E}$ spanned by the vertices $\{n,...,v+1\}$ its vertices, a simplex $\tau\in \Sigma$, an ordered list $N$ of vertices which is assumed to be the largest vertices of the simplices which have $\tau$ as a minimal $k$-face, and an table $L$ which contains the value of the largest neighbor of each vertex in the graph. Add-Cofaces then constructs a new simplicial complex $\Sigma'$ by adding to $\Sigma$  
the simplex of the form $\tau' = \tau \cup \{v\}$, for each $v \in N$, after which it constructs the new set $M\coloneqq \{w \in N \mid \{v,w\}\in E\}$ which will become the children of $v$ in the simplex tree. It then recursively calls itself on $\tau'$, $\Sigma'$, and the new set $M$.

\begin{algorithm}
	\caption{\texttt{New-Add-Cofaces($G=(V,E),d,\tau,N,\Sigma,L$)}}
	\begin{algorithmic}[1]
		\State $\Sigma \gets \Sigma \cup \{\tau\}$
		\If{$\dim (\tau)\geq d$}
		\State \Return
		\Else
		\ForAll{$v \in N$}
		\State $\sigma \gets \tau \cup \{v\}$
		\State $M \gets \texttt{Table-Lookup}(G,N,L,v)$ \label{line:Table-Lookup NAC}
		\State \texttt{New-Add-Cofaces($G,d,\sigma,M,\Sigma,L$)}
		\EndFor
		\EndIf
	\end{algorithmic}
\end{algorithm}

The \texttt{New-VR} algorithm now applies \texttt{New-Add-Cofaces} to each vertex $v$ of the graph $G$.

\begin{algorithm}
	\caption{\texttt{New-VR($G=(V,E),d)$}}
	\begin{algorithmic}
		\State $\Sigma \gets V \cup E$
		\ForAll {$u \in V$}
		\State $N \gets \texttt{Upper-Neighbors}(G,u)$
		\State \texttt{New-Add-Cofaces}($G,d,\{u\},N,\Sigma$)
		\EndFor
		\State \Return $\Sigma$
	\end{algorithmic}
\end{algorithm}

\begin{myalgorithm}
	\label{alg:Parallel}(Parallel)
	We parallelize the above algorithm by noting that the comparisons are all independent of one another, and once we have two simplices in \(D(k)\) which share a common \((k-1)\)-face, we may begin constructing \(D(k+1)\).
\end{myalgorithm}

\begin{theorem}
Given a graph \(G=(V,E)\), Algorithms \ref{alg:Serial} and \autoref{alg:Parallel} construct the clique-complex \(\Sigma_G\) of \(G\).
\end{theorem}
\begin{proof}
First, since Algorithm \autoref{alg:Parallel} is simply the parallelization of Algorithm \autoref{alg:Serial}, it is enough to prove that Algorithm \autoref{alg:Serial} constructs the clique complex of \(G\). This, however, follows from Theorems \ref{thm:Pair check} and \ref{thm:Main}, and the fact that, at every layer \(D(k)\) of the simplex tree diagram, the criterion in \autoref{thm:Pair check} is checked for every pair of \(k\)-simplices \((\sigma_0 < \sigma_1)\) which share a common \((k-1)\)-face \(\rho\).
\end{proof}

\section{The Zomorodian Algorithms and the Simplex Tree}
\label{sec:Background}
In this section, we describe the algorithms for constructing the Vietoris-Rips complex from \cite{Zomorodian_2010} and \cite{Boissonnat_Maria_2012} and discuss the relationships between them. 
We first give the definition of the Vietoris-Rips complex.

\begin{definition}
	Let $G=\{V,E\}$ be an undirected graph with vertices $V$ and edge set $E$. The Vietoris-Rips (or clique) complex $\Sigma$ of $G$ is the abstract simplicial complex given by
	\begin{equation*}
		\sigma = \{v_0,\dots,v_n\} \in \Sigma \iff \forall v_i,v_j\in \sigma, i \neq j, \{v_i,v_j\} \in E.
	\end{equation*}
	Note that, by definition, $\sigma \in \Sigma, \sigma' \subset \sigma \implies \sigma'\in \Sigma$, and therefore $\Sigma$ is a simplicial
	complex.
\end{definition}

We now recall the Inductive-VR construction from \cite{Zomorodian_2010} (substituting upper neighbors for lower neighbors for consistency with the notation in \cite{Boissonnat_Maria_2012}). The Inductive-VR method does the following. Starting with the graph, it adds the $2$-dimensional cofaces of each edge, then the $3$-dimensional
cofaces of each $2$-d face, and so on.

\begin{algorithm}
	\caption{\texttt{Inductive-VR($G=(V,E), d$)}}
\begin{algorithmic}
	\State $\Sigma \gets V \cup E$
	\For{$i=1$ to $d$}
		\For{\texttt{each $k$-simplex $\tau \in \Sigma$}}
			\State $N \gets \bigcap_{v \in \tau}$ \texttt{Upper-Neighbors($G,v$)}
			\For{each $w \in N$}
				\State $\Sigma \gets \Sigma \cup \{\tau \cup \{w\}\}$
			\EndFor
		\EndFor
	\EndFor
	\State \Return $\Sigma$	
\end{algorithmic}
\end{algorithm}

The next method, the Iterative-VR construction from \cite{Zomorodian_2010}, rests on the following observation: Given a $k$-simplex $\tau' = \tau \cup \{v\}$, where $\tau$ and $\tau'$ are both in the Vietoris-Rips complex, the $(k+1)$-dimensional cofaces of $\tau'$ such that all of their vertices have labels greater than $v$ are exactly the simplices of the form $\tau' \cup \{w\}$, where $w \in \tau \cap \text{Upper-Neighbors}(v)$. (Compare with Figure 2 in \cite{Boissonnat_Maria_2012}). 

The \texttt{Add-Cofaces} function for the Incremental-VR method is identical to the \texttt{New-Add-Cofaces} except for at Line 7 below. Note that the sets $M$ is each case will be identical, but the methods to calculate them are different. This function takes a simplicial complex $\Sigma$, assumed to be the clique complex of the subgraph of $G={V,E}$ spanned by the vertices $\{n,...,v+1\}$ its vertices, a simplex $\tau\in \Sigma$, and a collection $N$ of vertices which is assumed to be the intersection of the upper neighbors of each vertex in $\tau$. Add-Cofaces then constructs a new simplicial complex $\Sigma'$ by adding to $\Sigma$ all of the simplices of the form $\tau' = \tau \cup \{v\}$, for each $v \in N$, after which it recursively calls itself on $\tau'$, $\Sigma'$, and $M = N \cap \text{Upper-Neighors}(v)$. This function is summarized here:

\begin{algorithm}
	\caption{\texttt{Add-Cofaces($G=(V,E),d,\tau,N,\Sigma$)}}\label{alg:AC}
	\begin{algorithmic}[1]
		\State $\Sigma \gets \Sigma \cup \{\tau\}$
		\If{$\dim (\tau)\geq d$}
		\State \Return
		\Else \ForAll{$v \in N$}
		\State $\sigma \gets \tau \cup \{v\}$
		\State $M \gets N \cap \texttt{Upper-Neighbors}(G,v)$ \label{line:Intersection AC}
		\State \texttt{Add-Cofaces($G,d,\sigma,M,\Sigma$)}
		\EndFor
		\EndIf
	\end{algorithmic}
\end{algorithm}

The Incremental-VR algorithm then results from applying Add-Cofaces to each vertex $v \in V$ in the graph, i.e.

\begin{algorithm}
	\caption{\texttt{Incremental-VR}}
	\begin{algorithmic}
		\State $\Sigma \gets \Sigma \cup \{\tau\}$
		\ForAll {$u \in V$}
		\State $N \gets \texttt{Upper-Neighbors}(G,u)$
		\State \texttt{Add-Cofaces}($G,d,\{u\},N,\Sigma$)
		\EndFor
		\State \Return $\Sigma$
	\end{algorithmic}
\end{algorithm}

Finally, the simplex tree data structure, introduced in \cites{Boissonnat_Maria_2012}, encodes a simplicial complex as an directed tree (i.e. a tree with directed edges) where the edges are directed from parent nodes to child nodes. Each node of a simplex tree has a label which is one of the vertices of the graph, and multiple nodes may (and usually do) have the same label. A simplex $\{v_0,v_1,v_2,...,v_k\}$, where $v_0 < v_1 < \cdots < v_k$ is represented in the tree as a (directed) path. That is, if $\{v_0, v_1,v_2\}$ is a simplex in a simplicial complex $\Sigma$ and $v_0 < v_1 < v_2$, then $v_0$ is inserted at level $0$ of the corresponding simplex tree, $v_1$ is its child, and $v_2$ is the child of $v_2$, etc. which, in particular, gives a bijection between $k$-dimensional simplices in a simplicial complex and the nodes in the $k$-th level of the simplex tree (we define the level with vertices as level zero).  The algorithm to create the Vietoris-Rips complex of a graph using a simplex tree data structure is exactly the same as the Incremental-VR algorithm above, except that one is more specific about the data structure used to encode the simplicial complex, and the set of vertices $M$ is inserted into the simplex tree as children of $v$ at each step in the \texttt{Add-Cofaces} function above.

In the implementation of the Incremental-VR algorithm in the software package GUDHI, the intersection of the two sets in Line \ref{line:Intersection AC} of Algorithm \ref{alg:AC} is performed using the following common \texttt{merge} algorithm. Let $L_1$ and $L_2$ be two ordered lists. The specific intersection algorithm implemented in GUDHI v3.9.0 is found in Algorithm \ref{alg:merge}:

\begin{algorithm}
	\caption{\texttt{GUDHI Merge-Intersect}}\label{alg:merge}
\begin{algorithmic}[1]
	\State \texttt{$I$ = empty list}
	\State $v_1 \gets \texttt{first element of } L_1$
	\State $v_2 \gets \texttt{first element of } L_2$
	\While {$((v_1 \leq \texttt{end}(L_1)) \texttt{ \&\& } (v_2 \leq \texttt{end}(L_2)))$}
		\If {$(v_1 == v_2)$} 
		\State	$I \gets I \cup \{v_1\}$
		\State	$v_1 \gets \texttt{the next element in } L_1$		\State	$v_2 \gets \texttt{the next element in } L_2$
		\ElsIf{$((v_1 \geq$ \texttt{end}$(L_2) || (v_2 \leq$ \texttt{begin}$(L_1)))$}\label{line:Stopping cond 1}
		\State	\Return $I$
		\ElsIf{$(v_1 < v_2)$}
		\State $v_1 \gets$ \texttt{the next element in }$L_1$
		\ElsIf{$(((v_1 \leq \texttt{end}(L_2)) || (v_2 \geq \texttt{begin}(L_1))))$} \label{line:Stopping cond 2}
		\State	\Return $I$;
		\Else
		\State	$v_2 \gets$ \texttt{the next element in }$L_2$
		\EndIf
	\EndWhile
	\State \Return $I$
\end{algorithmic}
\end{algorithm}

Although not mentioned in \cites{Boissonnat_Maria_2012, Zomorodian_2010}, as in \ref{alg:Parallel} above, the sub-complexes $\Sigma_i$ which result by applying \texttt{Add-Cofaces} to vertices $v_i$ are independent of each other, and they only require access to the graph in order to construct. The incremental algorithm is therefore (embarrassingly) parallelizable, at least assuming that the original graph is available to all of the parallel processes. Since the graph is fixed, this may be easily implemented in a number of ways.
 
\section{The Complexity Difference Between the Incremental-VR and New-VR Algorithms}

We now briefly discuss give a simple example which illustrates the difference in behavior between the Incremental-VR (GUDHI) and our New-VR algorithms. Consider the graph in Figure \ref{fig:Graph}. Its clique complex, has $10$ vertices, $4$ faces, and a single $3$-dimensional simplex.
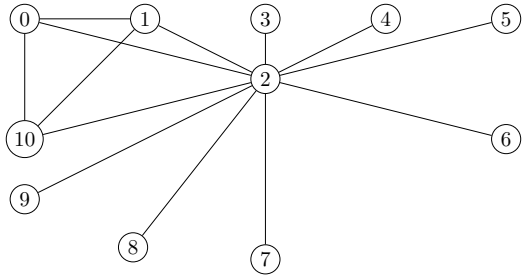
\begin{figure}[!ht]
	\centering
	\scalebox{0.8}{
\begin{tikzpicture}[every node/.style={draw,circle,inner sep=2pt}]
	\node (2) at (6,-1) {2};
	\node (0) at (2,0) {0};
	\node (1) at (4,0) {1};
	\node (3) at (6,0) {3};
	\node (4) at (8,0) {4};
	\node (5) at (10,0) {5};
	\node (6) at (10,-2) {6};
	\node (7) at (6,-4) {7};
	\node (8) at (3.8,-3.8) {8};
	\node (9) at (2,-3) {9};
	\node (10) at (2,-2) {10};

	\foreach \u/\v in {2/1,2/0,1/0,2/3,2/4,2/5,2/6,2/7,2/8,2/9,2/10,10/1,0/10}
	\draw (\u) -- (\v);
\end{tikzpicture}}
\caption{Graph which demonstrates the extra comparisons performed by \texttt{Add-Coface}. When adding node $10$ as a child to node $2$, the function \texttt{merge-intersect} will compare $10$ to all the nodes from $3$ to $10$. In \texttt{New-Add-Cofaces}, however, only a table lookup at the coordinates $(2,10)$ made, saving $7$ comparisons.}\label{fig:Graph}
\end{figure}
The computationally intensive part of both algorithms is the passage from one level to the next in the simplex tree. As we pass form level $2$ in the tree (the $2$-skeleton of simplicial complex) to level $3$ (the $3$-skeleton), in \texttt{Incremental-VR}, the algorithm computes the intersection $M \cap \text{Upper-Neighbors}(2)$ using \texttt{merge-intersect}. Since $M= \{10\}$ and \texttt{Upper-Neighbors}$(2)= \{3,\dots,10\}$, the Incremental-VR algorithm needlessly compares $\{2,3,\dots 8,9\}$ to $10$ before finally adding $v_10$ to the next level of the simplex tree. Conversely, in the New-VR algorithm, the only table look-up which is performed is for
the table entry at index $(2,10)$, resulting in faster performance.

We examine the relative computational complexity of the two algorithms by first noting that the only differences between them is that the function \texttt{GUDHI-Merge-Intersect} in line \ref{line:Intersection AC} of \texttt{Add-Cofaces} has been replaced by the function \texttt{Table-Lookup} in line \ref{line:Table-Lookup NAC} of \texttt{New-Add-Cofaces}, and so we focus on simply comparing these two functions. However, computing the relative average case computational complexity of both algorithms is somewhat complicated by lines \ref{line:Stopping cond 1} and \ref{line:Stopping cond 2} of the \texttt{GUDHI Merge-Intersect} function, which stops the comparisons by taking into consideration of the start and end points of the respective lists, and, similarly, including $L(v)$ in the stopping condition $w\leq \texttt{min}(\texttt{end}(N),L(v))$ in line \ref{line:stopping cond TL} of the \texttt{New-Add-Cofaces} function.  We will instead examine the complexity of simplified versions of these functions, which, while not exactly the implementations that were tested, are easier to analyze, and share the same basic behavior as seen in our experiments. For reference, the versions of these algorithms which we will consider in the following discussion are contained in Algorithms \ref{alg:Simp-Table-Lookup} and \ref{alg:simp-merge} below.

\begin{algorithm}
	\caption{\texttt{Simplified-Table-Lookup}($G,N,L,v$)}\label{alg:Simp-Table-Lookup}
	\begin{algorithmic}[1]
		\State $M \gets \texttt{empty list}$
		\ForAll{$w \in N$} \label{line:simple stopping cond TL}
		\If{$\{v,w\}\in E$}
		\State $M \gets M \cup \{w\}$
		\EndIf 
		\EndFor
		\State \Return M
	\end{algorithmic}
\end{algorithm}

\begin{algorithm}
	\caption{\texttt{Simplified-Merge-Intersect}}\label{alg:simp-merge}
	\begin{algorithmic}[1]
		\State \texttt{$I$ = empty list}
		\State $v_1 \gets \texttt{first element of } L_1$
		\State $v_2 \gets \texttt{first element of } L_2$
		\While {$((v_1 \leq \texttt{end}(L_1)) \texttt{ \&\& } (v_2 \leq \texttt{end}(L_2)))$}
		\If {$(v_1 == v_2)$} 
		\State	$I \gets I \cup \{v_1\}$
		\State	$v_1 \gets \texttt{the next element in } L_1$		\State	$v_2 \gets \texttt{the next element in } L_2$
		\ElsIf{$(v_1 < v_2)$}
		\State $v_1 \gets$ \texttt{the next element in }$L_1$
		\Else
		\State	$v_2 \gets$ \texttt{the next element in }$L_2$
		\EndIf
		\EndWhile
		\State \Return $I$
	\end{algorithmic}
\end{algorithm}

To examine the relative computational complexity of these two algorithms on Erd\H{o}s-Rényi graphs, consider a simplex $\sigma = \{v_0,\dots,v\} \in \Sigma$ be a simplex in the Vietoris-Rips complex of an Erd\H{o}s-Rényi graph $G(n,p)$. The complexity of \texttt{Simplified-Table-Lookup} is $O(S)$, where $S$ is the number of siblings of $v$, and therefore for the Erd\H{o}s-Rényi graph $G(n,p)$, the expected value of $S$ is $\mathbb{E}(S) = np^k$, where $k$ is the layer of the simplex tree in which $v$ resides, or, equivalently, the dimension of $\sigma$. Conversely, the expected value of the number of neighbors of any vertex is $np$, and the expected complexity of the \texttt{Simplified-Merge-Intersect} function is therefore $O(n(p + p^k)$. Furthermore, in the regime where $np_n = c$, $c$ a constant (preserving the expected degree of each vertex as $n$ increases), we have that $p = c/n$, and the critical step is better by an order of magnitude in \texttt{Simplified-Table-Lookup} versus \texttt{Simplified-Merge-Lookup}. (This does not imply that the entire algorithm is necessarily an order of magnitude better in this regime, however, only that it is better and that there is a marked difference in the performance of \texttt{Simplified-Table-Lookup} and \texttt{Simplified-Merge-Intersect}.)

 Since the number of operations in \texttt{Table-Lookup} function is at most two times that of the $\texttt{Simplified-Table-Lookup}$ function due to the additional verification, and may be smaller, the computational complexity of \texttt{Table-Lookup} is bounded above by $O(S)$. While we do not know the precise value of the computational complexity of \texttt{GUDHI-Merge-Intersect}, we conjecture that it is equal to that of \texttt{Simplified-Merge-Intersect}. In any event, while \texttt{Simplified-Merge-Intersect} may be replaced by a superior algorithm for the intersection of two lists, we are unaware of any such algorithm whose average complexity has been shown to be equal to or better in the size of the \emph{smaller} of the two lists $(i.e. \leq O(M))$), making the \texttt{New-VR} algorithm more efficient than any instance of the \texttt{Incremental-VR} algorithm using known methods for computing the intersection of two ordered lists. Furthermore, similar behavior may be expected in any Vietoris-Rips complex built from any random graph where the expected degree of a vertex is independent of the vertices, as the expected number of simplices with smallest vertex $v_0$ decreases exponentially with dimension, whereas the expected number of neighbors of any vertex $v_0$ is independent of dimension, making the \texttt{New-VR} algorithm superior in these cases.    
\section{Experiments}
\label{sec:org1b833da}

The \texttt{New-VR} algorithm and the experiments below were implemented in C/C++ using the clang and clang++ compilers, and run on an 2.4GHz Intel Xeon Silver 4214R CPU running Ubuntu Linux 20.04.06 LTS. The programs used the single-threaded
versions of the algorithms, and for the Incremental-VR algorithm, the implementation in GUDHI v3.9.0 was used.

For each combination of $p=\{0.1,\dots,0.5\}$ and $\text{Maximal Dimension}=\{2,\dots,5\}$, we generated $100$ Erd\H{o}s-Rényi graphs $G(n=125,p)$ and constructed their Vietoris-Rips complexes up to the maximal dimension for the experiment. (Note that both algorithms were tested against the same random graphs.) As one can see from Tables \ref{table:NewAlgMeans} and \ref{table:OldAlgMeans} below, the new algorithm studied in this article was, on average, $5$-$10$ times faster than the Incremental-VR algorithm, with higher increases for sparser graphs and in higher dimensions, as expected from the complexity analysis. We believe that the superior performance of the Incremental-VR algorithm seen in dimension two is principally due to better memory handling and other optimizations in the GUDHI codebase which are independent of the algorithm for computing the Vietoris-Rips complex itself.

\begin{table}[h!]
\begin{center}
\begin{tabular}{|c|c|c|c|c|c|}
	\hline
	\multicolumn{6}{|c|}{\textbf{Mean Run Times (Microseconds) for the \texttt{New-VR} Algorithm}}\\
	\hline
	$\mathbf{p}\backslash$\textbf{Dimension} & 2 & 3 & 4 & 5 & 6 \\
	\hline
	0.1 & 109.85 & 110.65 & 150.12 & 232.78 & 239.05 \\
	\hline
	0.2 & 409.04 & 396.52 & 452.65 & 623.84 & 781.28 \\
	\hline
	0.3 & 1,418.11 & 1,549.58 & 1,637.73 & 1,845.14 & 2,687.03 \\
	\hline
	0.4 & 4,521.96 & 6,620.21 & 7,327.05 & 8,195.19 & 9,705.48 \\
	\hline
	0.5 & 12,142.36 & 28,439.83 & 47,432.14 & 53,391.11 &  54,866.87 \\
	\hline
	0.6 & 27,855.32 & 115,005.09 & 322,749.59 & 543,911.66 & (No Data Available)\\
	\hline 
\end{tabular}
\end{center}
\caption{The mean run-times (in microseconds) for the Algorithm \ref{alg:Serial} above, tested on $100$ Erd\H{o}s-Rényi graphs $G(n=100,p)$, where the probability $p$ is given in the first column of the table. The maximal dimension of the complexes is given in the top row of the table.}	\label{table:NewAlgMeans}
\end{table}

\begin{table}[ht!]
\begin{center}
\begin{tabular}{|c|c|c|c|c|c|}
	\hline
	\multicolumn{6}{|c|}{\textbf{Mean Run Times for the Incremental-VR Algorithm (GUDHI v3.9.0)}}\\
	\hline
	$\mathbf{p}\backslash$\textbf{Dimension} & 2 & 3 & 4 & 5 & 6 \\
	\hline
	 0.1 & 1,081.24 & 1,142.45 & 1,213.76 & 1,280.47 & 1,345.54 \\
	 \hline
	 0.2& 2,853.89 & 3,892.48 & 4,369.18 & 4,693.18 & 4,662.37 \\
	 \hline
	 0.3 & 5,559.44 & 11,509.93 & 14,834.08 & 16,121.94 & 16,414.89 \\
	 \hline
	 0.4 & 9,230.05 & 31,173.14 & 56,338.33 & 68,455.99 & 71,643.62 \\
	 \hline
	0.5& 14,364.26 & 75,938.71 & 221,483.81 & 367,614.49 & 446,944.97 \\
	\hline
	0.6 & 21,133.60 & 716,949.53 & 829,931.88 & 2,259,241.63 & (No Data Available)\\ 
	\hline
\end{tabular}
\end{center}
	\caption{The mean run-times (in microseconds) for the Incremental-VR algorithm, tested on $100$ Erd\H{o}s-Rényi graphs $G(n=100,p)$, where the probability $p$ is given in the first column of the table. The maximal dimension of the complexes is given in the top row of the table.}
		\label{table:OldAlgMeans}
\end{table}

\section{Discussion and Future Work}
\label{sec:org4363633}

In this article, we presented a new method for constructing the Vietoris-Rips complex of a graph, exploiting a small amount of combinatorial structure of the complex to simplify building the $(k+1)$-skeleton of the complex from the $k$-dimensional skeleton. Since our algorithm constructs each \((k+1)\)-simplex in a clique complex using
the $k$-skeleton of the complex combined with a single additional verification at
only one pair of vertices, we conjecture that this represents the most efficient process possible, i.e. that
\begin{conjecture}
Algorithms \autoref{alg:Serial} and \autoref{alg:Parallel}
are optimal serial and parallel algorithms, respectively, for
constructing the $k$-skeleta of the clique complex of a graph.
\end{conjecture}
We leave the exact notion of optimality in this conjecture (worst-case
complexity, average complexity for specific random graph models, best practical
performance for 'real' examples, etc.) open to interpretation. We gave an initial comparison of the complexity of a simplified version of the algorithm to a simplified version of the \texttt{Incremental-VR} algorithm as implemented in GUDHI v3.9.0, established that the complexity of the simplified \texttt{New-VR} algorithm is an upper bound for the complexity of the original \texttt{New-VR} algorithm, and we conjecture that the complexity of both of the original algorithms is equal to that of their simplified counterparts.

Multi-threaded and GPU-accelerated implementations of the parallel
algorithm would be naturally expected to significantly accelerate the performance of the new algorithm, and these are currently being developed.

Finally, the experiments on the Erd\H{o}s-Rényi graphs show that the performance of the new algorithm in sparse regimes, even in high dimensions, is as good or better than the performance of the GUDHI v3.9.0 algorithm in dimension $2$ with different values of $p$. We are interested to know whether this improvement - or at least the parallelized versions - would be enough to make the calculation of higher-dimensional ($> 1$) persistent homology practical for interesting small- to medium-sized data sets.

\acknowledgement{We are grateful to Marc Glisse for pointing out the reference \cite{Boissonnat_Maria_2012}.}

\printbibliography
\end{document}